\documentclass[a4paper]{amsart}

\usepackage{amsmath,amssymb,amsthm,amscd}
\usepackage[matrix,arrow,curve,tips]{xy}
            \SelectTips{cm}{}

\usepackage{mathrsfs}

\newtheorem{dfn}{Definition}[section]

\newtheorem{prop}[dfn]{Proposition}
\newtheorem{theo}[dfn]{Theorem}
\newtheorem{ex}[dfn]{Example}

\newtheorem{rem}[dfn]{Remark}

\newcommand{\RR}{\mathbb{R}}
\newcommand{\CC}{\mathbb{C}}

\newcommand{\NN}{\mathbb{N}}
\newcommand{\ZZ}{\mathbb{Z}}

\newcommand{\cE}{\mathcal{E}}
\newcommand{\cF}{\mathcal{F}}

\newcommand{\Hom}{\text{Hom}}
\newcommand{\Lin}{\text{Lin}}

\newcommand{\com}{\mathbin{{\scriptstyle \circ }}}

\newcommand{\ev}{\mathord{\mathrm{ev}}}

\newcommand{\C}{\mathord{\mathcal{C}^{\infty}}}
\newcommand{\Cc}{\mathord{\mathcal{C}^{\infty}_{c}}}

\title[]
      {Automatic continuity of transversal distributions}

\author{J. Kali\v{s}nik}
\address{Department of Mathematics, University of Ljubljana,
         Jadranska 19, 1000 Ljubljana, Slovenia;
         Institute of Mathematics, Physics and Mechanics,
         University of Ljubljana, Jadranska 19,
         1000 Ljubljana, Slovenia}
\email{jure.kalisnik@fmf.uni-lj.si}

\thanks{This research was supported by research grants J1-1690, N1-0137 and research program Analysis and Geometry P1-0291 from 
the Slovenian Research Agency ARRS}


\begin{document}

\begin{abstract}
We prove that every transversal distribution on the fiber bundle $M\times N\to M$ is automatically continuous.
\end{abstract}

\maketitle

\section{Introduction}

The space of continuous linear operators on the Fr\'{e}chet algebra $\C(M)$ of smooth
functions on a smooth manifold $M$ contains many important examples, such as differential, integral
and translation operators. If $N$ is another smooth manifold, any such operator
$T:\C(N)\to\C(M)$ can be by the Schwartz kernel theorem 
represented by its distributional kernel. In \cite{Tre67} such kernels were called semiregular
in $M$. Each such kernel can be naturally considered as a smooth map on $M$ with values in the
space $\cE'(N)$ of compactly supported distributions on $N$ and in this way we obtain an isomorphism 
\[
\Hom(\C(N),\C(M))\cong\C(M,\cE'(N)).
\]

More recently such distributions have been studied in the context of distributions on Lie groupoids. 
In \cite{AnSk11} distributions which are transverse to a submersion were used to define 
pseudodifferential calculus on a singular foliation. In \cite{ErYu19} similar objects were used to define 
pseudodifferential calculus on a manifold from a geometric viewpoint in terms of the tangent groupoid. 
Distributions which are transverse to the source and the target map of a Lie groupoid form distributional 
convolution algebra of a Lie groupoid \cite{LeMaVa17},
which generalizes the convolution algebra $\cE'(G)$ of compactly supported distributions on a Lie group $G$.

The space $\Hom(\C(N),\C(M))$ can be also described from another point of view. By using the projection from
$M\times N$ to $M$, one can define a $\C(M)$-module structure on $\C(M\times N)$. The space 
$\Hom_{\C(M)}(\C(M\times N),\C(M))$ of continuous $\C(M)$-linear maps is then a direct generalization of 
the space $\cE'(N)$ and coincides with it if $M$ is a point. Elements of $\Hom_{\C(M)}(\C(M\times N),\C(M))$
will be called continuous transversal distributions on the bundle $M\times N\to M$. 
It was shown in \cite{LeMaVa17} that there is a natural isomorphism
\[
\C(M,\cE'(N))\cong\Hom_{\C(M)}(\C(M\times N),\C(M)).
\]
The main result in our paper is the following theorem.
\begin{theo}
Let $M$ and $N$ be smooth manifolds and let $\dim(M)>0$. Then we have the equality
\[
\Lin_{\C(M)}(\C(M\times N),\C(M))=\Hom_{\C(M)}(\C(M\times N),\C(M)).
\]
In other words: every $\C(M)$-linear homomorphism from $\C(M\times N)$ to $\C(M)$ is 
automatically continuous.
\end{theo}
Combining this result with previously known results we obtain an isomorphism
\[
\Hom(\C(N),\C(M))\cong\Lin_{\C(M)}(\C(M\times N),\C(M)).
\]
This result is somewhat surprising since the space $\Lin(\C(N),\C(M))$ of all linear maps is always strictly larger than the
subset $\Hom(\C(N),\C(M))$ of continuous maps. It is a consequence of the fact that in $\C(M\times N)$ we have enough space
to locally smoothly deform an element of $\C(N)$ to prove continuity.

\section{Preliminaries}

For the convenience of the reader and to fix the notations, we will first quickly collect some
basic facts about locally convex vector spaces. Detailed statements and proofs of most of the
results stated below can be found for example in \cite{KrMi97},\cite{Tre67}.

All our locally convex spaces will be complex and Hausdorff. For any locally
convex space $E$ we will denote by $E^{*}=\Lin(E,\CC)$ the space of all linear functionals on $E$ and by
$E'=\Hom(E,\CC)$ the space of all continuous linear functionals on $E$. If $F$ is another locally convex space,
we similarly define by $\Lin(E,F)$ and $\Hom(E,F)$ the spaces of all linear maps respectively continuous linear maps.
If $E$ and $F$ are modules over a complex algebra $A$, we will denote the corresponding spaces of module 
homomorphisms by $\Lin_{A}(E,F)$ and $\Hom_{A}(E,F)$.

A smooth curve in $E$ is a mapping $\gamma:\RR\to E$ for which all iterated derivatives exist. If $M$ is a smooth manifold,
a vector valued function $\alpha:M\to E$ is smooth if in local coordinates all partial derivatives exist
and are continuous. A map $T:E\to F$ is smooth if it maps smooth curves in $E$ to smooth curves in $F$. If $T$ is a linear
map, then it is smooth if and only if it is bounded. In particular, this implies that all continuous linear maps are smooth.
A map $\alpha:M\to E$ is scalarly smooth if the map $\phi\circ\alpha:M\to\CC$ is smooth for every $\phi\in E'$. If $\alpha$
is smooth it is also scalarly smooth, but the implication in the reverse direction does not always hold. It holds though in the
cases which are of most interest to us. In particular, if the space $E$ is complete, then every scalarly smooth function into $E$ is 
smooth.

We will be frequently dealing both with smooth functions and smooth vector valued functions on a manifold. To make a distinction,
we will denote by $f(x)\in\CC$ the value of a function $f:M\to\CC$ at $x$ and by $u_{x}\in E$ the value of $u:M\to E$ at $x$. It
will be convenient to think of $u:M\to E$ as a family $u=(u_{x})$ of elements of $E$, parametrised by $M$. The spaces of functions
and smooth functions on $M$ with values in $E$ will be denoted by $\cF(M,E)$ and $\C(M,E)$.

For any $E$ we will denote by $E'_{b}$ the continuous dual, equipped with the (strong) topology of uniform convergence on bounded subsets of $E$.
We have an injective map $\hat\,:E\to(E'_{b})'$, given by $\hat{x}(\phi)=\phi(x)$ for $x\in E$ and $\phi\in E'$. In general
this map is not an isomorphism. If $M$ is a smooth manifold, then we will denote by $\cE'(M)=\C(M)'_{b}$ the space of
compactly supported distributions on $M$, equipped with the strong topology. It is known that the space $\C(M)$ is reflexive, which
means that $\C(M)\cong\cE'(M)'_{b}$.

We will also consider the space $\cE^{*}(M)$ of all linear functionals on $\C(M)$ and equip it with the topology of pointwise convergence on
$\C(M)$. The space $\cE^{*}(M)$ is complete and we have $\cE^{*}(M)=\C(M)'_{\text{init}}$, 
where $\C(M)_{\text{init}}$ is the space $\C(M)$, equipped with the finest locally convex vector space topology. We then have an isomorphism 
of vector spaces $\cE^{*}(M)'\cong\C(M)$.

\section{Automatic continuity of transversal distributions}

Throughout this section $M$ and $N$ will be smooth, Hausdorff, second countable manifolds. We will furthermore assume that $\dim(M)>0$.
The space $\C(M\times N)$ can be equipped with the action of the algebra $\C(M)$, given by the formula
\[
(\rho\cdot f)(x,y)=\rho(x)f(x,y)
\]
for $\rho\in\C(M)$, $f\in\C(M\times N)$, $x\in M$ and $y\in N$. We will show that any linear homomorphism
of $\C(M)$-modules $u\in\Lin_{\C(M)}(\C(M\times N),\C(M))$ is automatically continuous, if we equip the spaces
of smooth functions with Fr\'{e}chet topologies. 

We will denote by $N_{x}:={x}\times N$ the fiber of the projection $M\times N\to M$ over $x\in M$
and by $I_{N_{x}}=\{f\in\C(M\times N)\,|\,f|_{N_{x}}=0\}$ the ideal of functions that vanish on $N_{x}$.
The authors have shown in \cite{LeMaVa17} that any $f\in I_{N_{x}}$ can be written in the form
\[
f=\rho_{1}\cdot f_{1}+\ldots+\rho_{k}\cdot f_{k}
\]
for some functions $f_{1},\ldots,f_{k}\in\C(M\times N)$ and $\rho_{1},\ldots,\rho_{k}\in\C(M)$ for which 
$\rho_{1}(x)=\ldots=\rho_{k}(x)=0$. Let us now fix $u\in\Lin_{\C(M)}(\C(M\times N),\C(M))$. 
If we denote by $\ev_{x}:\C(M)\to\CC$ the evaluation at the point $x$,
it now follows from $\C(M)$-linearity of $u$ that
\[
(\ev_{x}\circ u)(f)=u(\rho_{1}\cdot f_{1}+\ldots+\rho_{k}\cdot f_{k})(x)=\rho_{1}(x)u(f_{1})(x)+\ldots+\rho_{k}(x)u(f_{k})(x)=0
\]
for any $f\in I_{N_{x}}$. Now note that $\C(M\times N)/I_{N_{x}}\cong\C(N_{x})$ to see
that the map $\ev_{x}\circ u$ induces a linear map
\[
u_{x}:\C(N_{x})\to\CC
\]
for every $x\in M$. It is defined by $u_{x}(f)=u(\tilde{f})$ where $\tilde{f}\in\C(M\times N)$ is an arbitrary extension of $f\in\C(N_{x})$.
By using the identification $\cE^{*}(N_{x})\cong\cE^{*}(N)$ we can combine all these maps into a map $\Phi(u):M\to\cE^{*}(N)$, 
given by $\Phi(u)_{x}=u_{x}$. In this way we obtain an injective linear map
\[
\Phi:\Lin_{\C(M)}(\C(M\times N),\C(M))\to\cF(M,\cE^{*}(N)).
\]
In the sequel we will show that we can identify the image of $\Phi$ with the space of smooth maps $\C(M,\cE'(N))$ with values in continuous linear functionals
on $\C(N)$.

\begin{prop}\label{openness of discontinuous values}
The set $\{x\in M\,|\,u_{x}\text{ is discontinuous}\}$ is an open subset of $M$ for every $u\in\Lin_{\C(M)}(\C(M\times N),\C(M))$.
\end{prop}
\begin{proof}
Suppose $u_{x}\in\cE^{*}(N)$ is discontinuous for some $x\in M$ which lies in the boundary of the set $\{x\in M\,|\,u_{x}\text{ is discontinuous}\}$. 
Since $\dim(M)>0$, we can find a sequence of points $(x_{n})$ in $M$, which converges to $x$, and such that $u_{x_{n}}\in\cE'(N)$ for all $n\in\NN$.

Let us now choose an arbitrary $f\in\C(N)$ and define $\tilde{f}\in\C(M\times N)$ by $\tilde{f}=f\com\text{pr}_{N}$. From definition
of $u$ it follows that $u(\tilde{f})$ is smooth on $M$ which implies that 
\[
u_{x}(f)=u(\tilde{f})(x)=\lim_{n\to\infty}u(\tilde{f})(x_{n})=\lim_{n\to\infty}u_{x_{n}}(f).
\]
The sequence of continuous distributions $(u_{x_{n}})$ therefore converges pointwise to $u_{x}$ in $\cE^{*}(N)$. Since $\C(N)$ is a barelled 
locally convex space, it follows from a version of Banach-Steinhaus theorem \cite{Tre67} that $u_{x}$ is continuous which leads us to contradiction.
\end{proof}
\newpage

\begin{prop}\label{emptyness of discontinuous values}
The set $\{x\in M\,|\,u_{x}\text{ is discontinuous}\}$ is empty for every $u\in\Lin_{\C(M)}(\C(M\times N),\C(M))$.
\end{prop}
\begin{proof}
Suppose $u_{x}\in\cE^{*}(N)$ is discontinuous for some $x\in M$. By Proposition \ref{openness of discontinuous values} we can find an 
open neighbourhood $U$ of $x$ such that $u_{t}$ is discontinuous for all $t\in U$. Let us choose a smooth embedding
$\iota:\RR/2\ZZ\to U$ such that $\iota(0)=x$ and denote $x_{n}=\iota(\frac{1}{n})$. In particular, this implies that $u_{x_{n}}$
is discontinuous for all $n\in\NN$.

Since $\C(N)$ is a Fr\'{e}chet space, we can find an increasing sequence of seminorms $(p_{n})$ on $\C(N)$, which generate
the Fr\'{e}chet topology. A discontinuous linear functional on a Fr\'{e}chet space is unbounded, so we can for each $n\in\NN$ find a function
$f_{n}\in\C(N)$ such that:
\begin{align*}
p_{n}(f_{n})&\leq\tfrac{1}{2^{n}}, \\
u_{x_{n}}(f_{n})&=1.
\end{align*}

We will now show that the sequence $(f_{n})$ converges fast to $0$ in $\C(N)$ (this means that for each $k\in\NN$ the sequence $(n^{k}f_{n})$
remains bounded). So let us fix $k,l\in\NN$ and suppose $l\leq n$. We then have 
\[
p_{l}(n^{k}f_{n})=n^{k}p_{l}(f_{n})\leq n^{k}p_{n}(f_{n})\leq\tfrac{n^{k}}{2^{n}}.
\]
This implies that for any seminorm $p_{l}$ the set $\{p_{l}(n^{k}f_{n})\,|\,n\in\NN\}$ is bounded in $\RR$ and as a result
the sequence $(n^{k}f_{n})$ is bounded in $\C(N)$.

By Special Curve Lemma \cite{KrMi97} there exists a smooth function $F:\RR/2\ZZ\to\C(N)$ such that $F_{0}=0$ and $F_{\frac{1}{n}}=f_{n}$.
Let us denote by $S=\iota(\RR/2\ZZ)$ the embedded circle in $M$. We then have the following isomorphisms
\[
\C(\RR/2\ZZ,\C(N))\cong\C(\RR/2\ZZ\times N)\cong\C(S\times N).
\]
Using these isomorphisms we obtain from $F$ a smooth function $f\in\C(S\times N)$ which satisfies $f|_{N_{x}}=0$ and $f|_{N_{x_{n}}}=f_{n}$.
Since $S\times N$ is a closed submanifold of $M\times N$, we can smoothly extend $f$ to a function $\tilde{f}\in\C(M\times N)$. 

Now note that the function $u(\tilde{f})$ satisfies:
\begin{align*}
u(\tilde{f})(x)&=0, \\
u(\tilde{f})(x_{n})&=1.
\end{align*}
Since $(x_{n})$ converges to $x$, this implies that $u(\tilde{f})$ is not continuous, which contradicts the assumption that $u(\tilde{f})\in\C(M)$.
\end{proof}

From Proposition \ref{emptyness of discontinuous values} it now follows that we have an injective linear map
\[
\Phi:\Lin_{\C(M)}(\C(M\times N),\C(M))\to\cF(M,\cE'(N)).
\]
We will next show that $\Phi(u)$ is smooth for all $u\in\Lin_{\C(M)}(\C(M\times N),\C(M))$.

\begin{prop}\label{Isomorphism between transversal distributions and smooth maps}
We have an isomorphism of vector spaces
\[
\Phi:\Lin_{\C(M)}(\C(M\times N),\C(M))\to\C(M,\cE'(N)).
\]
\end{prop}
\begin{proof}
We will first show that $\Phi$ maps into $\C(M,\cE'(N))$. Since the space $\cE'(N)$ is complete, it suffices to show that the map 
$\Phi(u):M\to\cE'(N)$ is scalarly smooth for every $u\in\Lin_{\C(M)}(\C(M\times N),\C(M))$.
The space $\C(N)$ is a reflexive locally convex space, which means that we can identify
$\C(N)$ with the continuous dual of $\cE'(N)$. For arbitrary $\hat{f}\in\cE'(N)'\cong\C(N)$ we now obtain the function 
$\hat{f}\circ\Phi(u):M\to\CC$ given by
\[
(\hat{f}\circ\Phi(u))(x)=\hat{f}(u_{x})=u_{x}(f)=u(f\com\text{pr}_{N})(x).
\]
Since $f\com\text{pr}_{N}\in\C(M\times N)$ it follows that $u(f\com\text{pr}_{N})\in\C(M)$, which shows that $\Phi(u)$ is scalarly smooth
and therefore smooth. 

We now know that $\Phi:\Lin_{\C(M)}(\C(M\times N),\C(M))\to\C(M,\cE'(N))$ is a well defined injective linear map. It remains to be proven that
it is surjective. Take any $\overline{u}\in\C(M,\cE'(N))$ and define for $f\in\C(M\times N)$ a map $u(f):M\to\CC$ by
\[
u(f)(x)=\overline{u}_{x}(f|_{N_{x}}).
\]
We will show that $u(f)$ is smooth. We can view function $f$ as a smooth family $f_{x}=f|_{N_{x}}$
of smooth functions on $N$, parametrized by $M$, by using the isomorphism 
$\C(M\times N)\cong\C(M,\C(N))$. Now observe that we can express $u(f)$ as the composition
\[
M\overset{(\overline{u},f)}{\rightarrow} \cE'(N)\times\C(N)\overset{\text{ev}}{\rightarrow}\CC,
\]
where $\text{ev}(\overline{u}_{x},f_{x})=\overline{u}_{x}(f_{x})$ is the natural pairing between the space and its dual. The map $\text{ev}$ is bilinear and
separately continuous. Since both factors are complete, it follows from Theorem $5.19$ in \cite{KrMi97} that $\text{ev}$ is bounded and therefore
smooth by Lemma $5.5$ in \cite{KrMi97}. The function $u(f)$ is then smooth as a composition of two smooth mappings.
It is now easy to check that the map $u:\C(M\times N)\to\C(M)$, defined as above, is $\C(M)$-linear and that it satisfies
the equality $\Phi(u)=\overline{u}$.
\end{proof}

By Proposition \ref{Isomorphism between transversal distributions and smooth maps} any $u\in\Lin_{\C(M)}(\C(M\times N),\C(M))$ can be defined
by a smooth family $\Phi(u)\in\C(M,\cE'(N))$ such that
\[
u(f)(x)=\Phi(u)_{x}(f|_{N_{x}}).
\]
On the other hand, the authors have shown in \cite{LeMaVa17} that any $u$, defined by a smooth family of distributions, is actually continuous, if
we equip the function spaces with Fr\'{e}chet topologies. For completeness we provide below the idea of the proof.

\begin{prop}[Lescure,Manchon,Vassout]\label{Smooth family defines continuous transversal distribution}
For any $\Phi(u)\in\C(M,\cE'(N))$ the $\C(M)$-linear map $u:\C(M\times N)\to\C(M)$, given by 
\[
u(f)(x)=u_{x}(f|_{N_{x}}),
\]
is continuous with respect to Fr\'{e}chet topologies.
\end{prop}
\begin{proof}
Let us assume for simplicity that $M=\RR^{k}$ and $N=\RR^{l}$ to avoid some technical complications. In this case the Fr\'{e}chet
topology on $\C(M)$ is generated by the sequence of seminorms $(p_{M,n})_{n\in\NN}$, given by
\[
p_{M,n}(\rho)=\sup_{\substack{|x|\leq n\\ |\alpha|\leq n}}\left|(D^{\alpha}\rho)(x)\right|,
\]
where $\alpha=(\alpha_{1},\ldots,\alpha_{k})\in\NN_{0}^{k}$, $|\alpha|=\alpha_{1}+\ldots+\alpha_{k}$
and $D^{\alpha}=\frac{\partial^{|\alpha|}}{\partial x_{1}^{\alpha_{1}}\partial x_{2}^{\alpha_{2}}\cdots\partial x_{k}^{\alpha_{k}}}$.
We need to show that any $\C(M)$-linear map $u:\C(M\times N)\to\C(M)$ is bounded. Take any $f\in\C(M\times N)$ 
and represent it as a smooth map $f:M\to\C(N)$.
We can then express $u(f)(x)=u_{x}(f_{x})$ to obtain the following generalised Leibniz rule 
\[
D^{\alpha}(u(f))(x)=\sum_{\beta\leq\alpha}\binom{\alpha}{\beta}(D^{\beta}u)_{x}\left((D^{\alpha-\beta}f)_{x}\right).
\]
From this we obtain the bound
\[
p_{M,n}(u(f))=\sup_{\substack{|x|\leq n\\ |\alpha|\leq n}}
\left|\sum_{\beta\leq\alpha}\binom{\alpha}{\beta}(D^{\beta}u)_{x}\left((D^{\alpha-\beta}f)_{x}\right)\right|
\leq C_{1}\sup_{\substack{|x|\leq n\\ |\alpha|\leq n\\ \beta\leq\alpha}}
\left|(D^{\beta}u)_{x}\left((D^{\alpha-\beta}f)_{x}\right)\right|
\]
for some constant $C_{1}$. Since $u:M\to\cE'(N)$ is continuous, the set
\[
E_{n}=\{(D^{\beta}u)_{x},|\,|x|\leq n,|\beta|\leq n\}
\]
is a compact subset of $\cE'(N)$. It is then bounded and hence equicontinuous by the Banach-Steinhaus theorem. 
As a result there exist a seminorm $p_{N,m}$ on $\C(N)$ and a constant $C_{2}$ such that
\[
|v(g)|\leq C_{2}\,p_{N,m}(g)
\]
for all $g\in\C(N)$ and all $v\in E_{n}$. Now note that
\[
\sup_{\substack{|x|\leq n\\ |\alpha|\leq n\\ \beta\leq\alpha}}
\left|(D^{\beta}u)_{x}\left((D^{\alpha-\beta}f)_{x}\right)\right|
\leq C_{2}\sup_{\substack{|x|\leq n\\ |\alpha|\leq n\\ \beta\leq\alpha}}p_{N,m}((D^{\alpha-\beta}f)_{x})
\leq C_{2} p_{M\times N,m+n}(f).
\]
to obtain the bound
\[
p_{M,n}(u(f))\leq C_{1}C_{2}\,p_{M\times N,m+n}(f)
\]
which shows that $u$ is bounded and hence continuous.
\end{proof}

As a result we obtain the following theorem.

\begin{theo}\label{theo_Main theorem}
Let $M$ and $N$ be smooth manifolds and let $\dim(M)>0$. Then we have the equality
\[
\Lin_{\C(M)}(\C(M\times N),\C(M))=\Hom_{\C(M)}(\C(M\times N),\C(M)).
\]
In other words: every $\C(M)$-linear homomorphism from $\C(M\times N)$ to $\C(M)$ is 
automatically continuous.
\end{theo}
Combining this theorem with the Schwartz kernel theorem we thus obtain for manifolds $M$ 
with $\dim(M)>0$ the isomorphism
\[
\Hom(\C(N),\C(M))\cong\Lin_{\C(M)}(\C(M\times N),\C(M)).
\]
\begin{rem}\rm
\begin{enumerate}
\item In a similar way we can prove a version of the Theorem \ref{theo_Main theorem} with compact supports. 
There are isomorphisms 
\[
\Hom(\C(N),\Cc(M))\cong\Cc(M,\cE'(N))\cong\Lin_{\Cc(M)}(\C(M\times N),\Cc(M)).
\]

\item If $\dim(M)=0$, for example if $M$ is a point, the Theorem \ref{theo_Main theorem} does not hold since there exist
discontinuous linear functionals on $\C(N)$. There is however no restriction on the dimension of $N$.
\end{enumerate}
\end{rem}

\section{Elements of $\Lin(\C(N),\C(M))$ as smooth families}

In this section we will for comparison briefly describe the space $\Lin(\C(N),\C(M))$ in terms of smooth
vector valued functions on $M$.

Take any $u\in\Lin(\C(N),\C(M))$. For every $x\in M$ we can define a linear map $u_{x}=\ev_{x}\circ u\in\cE^{*}(N)$, given by
\[
u_{x}(f)=u(f)(x)
\]
for $f\in\C(N)$. In this way we obtain an injective linear map
\[
\Phi:\Lin(\C(N),\C(M))\to\cF(M,\cE^{*}(N)).
\]
We will equip the space $\cE^{*}(N)$ with the topology of pointwise convergence to make it a complete locally convex vector space.

\begin{prop}\label{Linear maps}
We have an isomorphism of vector spaces
\[
\Phi:\Lin(\C(N),\C(M))\to\C(M,\cE^{*}(N)).
\]
\end{prop}
\begin{proof}
We will first show that the map $\Phi(u):M\to\cE^{*}(N)$ is smooth for any $u\in\Lin(\C(N),\C(M))$. Since $\cE^{*}(N)$ is complete, it
is enough to show that it is scalarly smooth. For any $f\in\C(M)\cong\cE^{*}(N)'$ we have 
\[
(\hat{f}\com\Phi(u))(x)=\hat{f}(u_{x})=u_{x}(f)=u(f)(x).
\]
Since $u(f)\in\C(M)$, it follows that $\Phi(u)$ is smooth. This shows that $\Phi$ maps $\Lin(\C(N),\C(M))$ into $\C(M,\cE^{*}(N))$.

We still have to show that $\Phi$ is surjective. Take any $\overline{u}\in\C(M,\cE^{*}(N))$ and any $f\in\C(N)$. Since $\hat{f}$
is a continuous linear functional on $\cE^{*}(N)$ it is smooth, so the mapping 
$\hat{f}\com\overline{u}:M\to\CC$ is smooth. If we define $u:\C(N)\to\C(M)$ by $u(f)=\hat{f}\com\overline{u}$,
we have $\Phi(u)=\overline{u}$.
\end{proof}

\begin{ex}\rm
Take any discontinuous $u_{0}\in\cE^{*}(N)$ and any nonzero $\rho\in\C(M)$. The family $u\in\C(M,\cE^{*}(N))$, given by
\[
u_{x}=\rho(x)u_{0}
\]
for $x\in M$ then defines a nontrivial example of a noncontinuous linear map in $\Lin(\C(N),\C(M))$.
\end{ex}

\noindent
{\bf Acknowledgements:} 
I would to thank O. Dragičević and F. Forstnerič for support and encouragement during research. I would also like to thank J. Mrčun
for many helpful discussions related to this topic.

\end{document}